\documentclass[11pt]{amsart}
\usepackage{amsmath}
\usepackage{amssymb}
\usepackage{amscd}

\def\NZQ{\mathbb}               
\def\NN{{\NZQ N}}

\def\RR{{\NZQ R}}
\def\CC{{\NZQ C}}

\newtheorem{Theorem}{Theorem}[section]

\newtheorem{Definition}[Theorem]{Definition}

\let\epsilon\varepsilon
\let\phi=\varphi
\let\kappa=\varkappa

\begin{document}

\title{Rectilinearization of sub analytic sets as a consequence of local monomialization}
\author{Steven Dale Cutkosky}
\thanks{partially supported by NSF}

\address{Steven Dale Cutkosky, Department of Mathematics,
University of Missouri, Columbia, MO 65211, USA}
\email{cutkoskys@missouri.edu}

\begin{abstract} 
We give a new proof of the rectilinearization theorem of Hironaka. We deduce rectilinearization as a consequence of our theorem on local monomialization for complex and real analytic morphisms.
\end{abstract}

\maketitle

In this paper we  deduce Hironka's rectilinearization theorem  \cite{H}  as an application of our theorem on local monomialization for complex and real analytic morphisms \cite{C}.

\begin{Definition} Suppose that $\phi:Y\rightarrow X$ is a  morphism of  complex or real analytic manifolds, and $p\in Y$. We will say that the map $\phi$ is monomial at $p$ if there exist regular parameters $x_1,\ldots,x_m,x_{m+1},\ldots,x_t$ in $\mathcal O_{X,\phi(p)}^{\rm an}$ and $y_1,\ldots,y_n$ in $\mathcal O_{Y,p}^{\rm an}$ and $c_{ij}\in \NN$ such that
$$
\phi^*(x_i)=\prod_{j=1}^ny_j^{c_{ij}}\mbox{ for }1\le i\le m
$$
with $\mathop{rank}(c_{ij})=m$ and $\phi^*(x_i)=0$ for $m<i\le t$. 

We will say that $\phi$ is monomial on $Y$ if there exists an open cover of $Y$ by open sets $U_k$ which are isomorphic to open subsets of $\CC^n$ (or $\RR^n$) and an open cover of $X$ by open sets $V_k$ which are isomorphic to open subsets of $\CC^t$ (or $\RR^t$) such that $\phi(U_k)\subset V_k$ for all $i$ and there exist
$c_{ij}(k)\in\NN$ such that
$$
\phi^*(x_i)=\prod_{j=1}^ny_j^{c_{ij}(k)}\mbox{ for }1\le i\le m
$$
with $\mathop{rank}(c_{ij})=m$ and $\phi^*(x_i)=0$ for $m< i\le t$, and 
where $x_i$ and $y_j$ are the respective coordinates on $\CC^t$ and $\CC^n$ (or $\RR^t$ and $\RR^n$). \end{Definition}

A local blow up of an  analytic space $X$  is a morphism $\pi:X'\rightarrow X$ determined by a triple $(U,E,\pi)$ where $U$ is an open subset of $X$, $E$ is a closed  analytic subspace  of $U$ and $\pi$ is the composition of the inclusion of $U$ into $X$ with the blowup of $E$.

Hironaka \cite{Hcar} and \cite{H} introduced the notion of an \'etoile on a complex analytic
space $Y$ to generalize a valuation of a function field of an algebraic
variety. On any  local blowup $Y_1$ of $Y$, an \'etoile $e$ determines a unique point on $Y_1$ called the
center of $e$.
We say that a sequence of local blowups $Y_1$ of a complex analytic space $Y$ is in an \'etoile $e$ on $Y$ if $e$ has a center on $Y_1$.

We now state the local monomialization theorem for complex analytic morphisms.

\begin{Theorem}\label{TheoremLM} (Local Monomialization, Theorem 1.2 \cite{C}) Suppose that $\phi:Y\rightarrow X$ is a morphism of reduced complex analytic spaces  and $e$ is an \'etoile over $Y$. Then there exists a commutative diagram of complex analytic morphisms
$$
\begin{array}{ccc}
Y_{e}&\stackrel{\phi_{e}}{\rightarrow}&X_e\\
\beta\downarrow &&\downarrow \alpha\\
Y&\stackrel{\phi}{\rightarrow}& X
\end{array}
$$
such that $\beta\in e$, the morphisms $\alpha$ and $\beta$ are  finite products of local blow ups of nonsingular analytic sub varieties, $Y_e$ and $ X_e$ are nonsingular analytic  spaces and $\phi_e$ is a  monomial  analytic morphism.

There exists a nowhere dense closed analytic subspace $F_e$ of $X_e$ such that $X_e\setminus F_e\rightarrow X$ is an open embedding and $\phi_e^{-1}(F_e)$ is nowhere dense in $Y_e$.
\end{Theorem}

Local monomialization theorems for real analytic morphisms  are also proven in \cite{C}. Local monomialization  along an arbitrary valuation is proven for morphisms of algebraic varieties in characteristic zero in \cite{Ast}, \cite{C3}  and \cite{LMTE}. Counterexamples to local monomialization for a morphism of characteristic $p>0$ algebraic varieties is given in \cite{C4}.

In  Theorem \ref{TheoremRdup} below, we show that  Hironaka's rectilinearization theorem, which was originally proven in \cite{H} can be deduced from local monomializiation,  Theorem \ref{TheoremLM}. Besides local monomialization,  our proof uses 
complexification of real analytic morphisms (Section 1 \cite{H}), resolution of singularities of analytic spaces (Hironaka \cite{H0}, \cite{HRes},  Aroca, Hironaka and Vicente \cite{AHV} and Bierstone and Milman \cite{BM1}), the Tarski Seidenberg Theorem, 
the fact that rectilinearization is true for semi analytic sets (this is a consequence of resolution of singularities, Hironaka \cite{H}) and the fact that the natural map from the vo\^ute \'etoil\'ee  of a complex analytic space $X$ to $X$ is proper (Hironaka, \cite{Hcar} and \cite{H}). 

Hironaka's proof (in Theorem 7.1 \cite{H}) makes essential use of the  local  flattening theorem (Hironaka,  Lejeune and Teissier \cite{HLT} and Section 4 of \cite{H}) and the fiber cutting lemma (Lemma 7.3.5 \cite{H}) to reduce to consideration  of proper finite morphisms.  In our proof, these arguments are replaced by the local monomialization theorem, Theorem \ref{TheoremLM}, and the reductions of Theorems \ref{TheoremA} and Theorem \ref{TheoremB} of this paper.

Some other notable proofs of rectilinearization are by Denef and Van Den Dries \cite{DV},  Bierstone and Milman \cite{BM3} and Parusinski \cite{aP}.  Hironaka deduces Lojasiewicz's inequalities for sub analytic sets  from rectilinearization in \cite{H}.

The related concept (to monomialization) of  toroidalization  for morphisms of algebraic varieties (\cite{AK} and \cite{ADK}) is used by Denef to prove $p$-adic quantifier elimination and related results \cite{De}, \cite{De1}. 
Some papers  on topics related to this article are   Teissier \cite{T},  Cano \cite{Ca},   Panazzolo \cite{P},  Lichtin, \cite{Li} and \cite{Li2}, Belotto \cite{B1} and Belotto, Bierstone, Grandjean and Milman \cite{BBGM}.

\begin{Theorem}\label{TheoremRdup}(Rectilinearization)  Let $X$ be a smooth connected real analytic space and let $A$ be a sub analytic subset of $X$. Let $p\in X$ and let $n=\dim X$. Then there exist a finite number of real analytic morphisms $\pi_{\alpha}:V_{\alpha}\rightarrow X$
which are finite sequences of local blowups over $X$ and  induce an open embedding of an open dense subset of $V_{\alpha}$ into $X$ such that:
\begin{enumerate}
\item[1)] Each $V_{\alpha}$ is isomorphic to $\RR^n$,
\item[2)] There exist compact neighborhoods $K_{\alpha}$ in $V_{\alpha}$ such that $\cup_{\alpha}(K_{\alpha})$ is a compact neighborhood of $p$ in $X$,
\item[3)] For each $\alpha$, $\pi_{\alpha}^{-1}(A)$ is union of quadrants in $\RR^n$.
\end{enumerate}
\end{Theorem}

Semi analytic and sub analytic sets in a real analytic space are defined in Section \ref{SecSub}.

A subset $B$ of $\RR^n$ is a quadrant if there exists a partition $\{1,\ldots,n\}=I_0\cup I_1\cup I_{-}$ such that $B$ is the set of $x\in \RR^n$ such that $x_i=0$ for all $i\in I_0$, $x_j>0$ for all $j\in I_+$ and $x_k<0$ for all $k\in I_{-}$ where 
$x_1,\ldots,x_n$ are the natural coordinates of $x$ in $\RR^n$.

We thank Jan Denef for suggesting  rectilinearization as an application of  local monomialization of  analytic morphisms, and for discussion and  encouragement. We also thank Bernard Teissier for discussions on this and related problems.

\section{Semi analytic and sub analytic sets}\label{SecSub}

We review the definitions of semi analytic and sub analytic sets from Chapter 6 \cite{H} (also see Chapter 1 \cite{BM3} or \cite{BM2}).

Let $X$ be a set and $\Delta$ be a family of subsets of $X$. The elementary closure $\tilde\Delta$ of $\Delta$ is the smallest family of subsets of $X$ containing $\Delta$ which is stable under finite intersection, finite union and complement.

Suppose that $U$ is an open subset of a real analytic space $X$.

Let $\Delta_+(U)$ be the set of subsets $A$ of $U$ of the form $A=\{x\in U\mid f(x)>0\}$ for some real analytic function $f$ on $U$. A subset $A$ of $X$ is said to be semi analytic at $x_0\in X$ if there exists an open neighborhood $U$ of $x_0$ in $X$ such that $A\cap U$ belongs to the elementary closure of $\Delta_+(U)$. $A$ is said to be semi analytic in $X$ if it is semi analytic at every point of $X$.

Let $\Gamma(U)$ be the set of those closed subsets of $U$ which are images of proper real analytic maps $g:Y\rightarrow U$. A subset $A$ of $X$ is said to be sub analytic at $x_0\in X$ if there exists an open neighborhood $U$ of $x_0$ in $X$ such that $A\cap U$ belongs to the elementary closure of $\Gamma(U)$. $A$ is said to be sub analytic in $X$ if it is sub analytic at every point of $X$.

\section{Preliminaries on  \'etoiles and local blow ups}\label{Pre}

We require that an analytic space be Hausdorff.

An \'etoile is defined in Definition 2.1 \cite{Hcar}. An \'etoile $e$ over a complex analytic space $X$ is defined as a subcategory of the category of sequences of local blow ups $\mathcal E(X)$ over $X$. If $\pi:X'\rightarrow X\in e$ then a point $e_{X'}\in X'$ is associated to $e$. We will call $e_{X'}$ the center of $e$ on $X'$. The \'etoile associates a point $e_X\in X$ to $X$ and if $\pi_1:X_1\rightarrow U$ is a local blow up of $X$ such that $e_X\in U$ then $\pi_1\in e$ and $e_{X_1}\in X_1$ satisfies $\pi_1(e_{X_1})=e_X$. If $\pi_2:X_2\rightarrow U_1$ is a local blow up of $X_1$ such that $e_{X_1}\in U_1$ then $\pi_1\pi_2\in e$ and $e_{X_2}\in X_2$ satisfies $\pi_2(e_{X_2})=e_{X_1}$. Continuing in this way, we can construct sequences of local blow ups
$$
X_n\stackrel{\pi_n}{\rightarrow} X_{n-1}\rightarrow \cdots\rightarrow X_1\stackrel{\pi_1}{\rightarrow} X
$$
such that $\pi_1\cdots\pi_i\in e$, with associated points $e_{X_i}\in X_i$ such that $\pi_i(e_{X_i})=e_{X_{i-1}}$ for all $i$.

 Let $X$ be a complex analytic space. Let $\mathcal E_X$ be the set of all \'etoiles over $X$ and for
$\pi:X_1\rightarrow X$ a product of local blow ups, let 
$$
\mathcal E_{\pi}=\{e\in \mathcal E_X\mid \pi\in e\}.
$$
Then the $\mathcal E_{\pi}$ form a basis for a topology on $\mathcal E_X$. The space $\mathcal E_X$ with this topology is called the vo\^ ute \'etoil\'ee over $X$ (Definition 3.1 \cite{Hcar}). The vo\^ ute \'etoil\'ee is a generalization to complex analytic spaces of the Zariski Riemann manifold of a variety $Z$ in algebraic geometry (Section 17, Chapter VI \cite{ZS2}).

We have a canonical map $P_X:\mathcal E_X\rightarrow X$ defined by $P_X(e)=e_X$ which is continuous, surjective and proper (Theorem 3.4 \cite{Hcar}).
It is shown in Section 2 of \cite{Hcar} that given a product of local blow ups $\pi:X_1\rightarrow X$, there is a natural homeomorphism
$j_{\pi}:\mathcal E_{X_1}\rightarrow \mathcal E_{\pi}$ giving a commutative diagram
$$
\begin{array}{rccl}
\mathcal E_{X_1}&\cong \mathcal E_{\pi}&\subset &\mathcal E_X\\
P_{X_1}\downarrow&&&\downarrow P_X\\
X_1&&\stackrel{\pi}{\rightarrow}&X.
\end{array}
$$

The join  of $\pi_1,\pi_2\in \mathcal E(Y)$ is defined in Proposition 2.9 \cite{Hcar}. The join is a morphism $J(\pi_1,\pi_2):Y_J\rightarrow Y$. 
It has the following universal property: Suppose that $f:Z\rightarrow Y$ is a strict morphism (Definition 2.1 \cite{Hcar}). Then there exists a $Y$-morphism 
$Z\rightarrow Y_J$ if and only if there exist $Y$-morphisms $Z\rightarrow Y_1$ and $Z\rightarrow Y_2$.
It follows from 2.9.2 \cite{Hcar} that if $\pi_1,\pi_2,\in e\in \mathcal E_Y$, then $J(\pi_1,\pi_2)\in e$. 
We describe the construction of Proposition 2.9 \cite{Hcar}. In the case when $\pi_1$ and $\pi_2$ are each local blowups,
which are described by the data $(U_i,E_i,\pi_i)$, $J(\pi_1,\pi_2)$ is the blow up 
$$
J(\pi_1,\pi_2):Y_J=B(\mathcal I_{E_1}\mathcal I_{E_2}\mathcal O_Y|U_1\cap U_2)\rightarrow Y.
$$
Now suppose that $\pi_1$ is a product $\alpha_0\alpha_{1}\cdots \alpha_r$ where $\alpha_i:Y_{i+1}\rightarrow Y_i$ are local blow ups defined by the data
$(U_i,E_i,\alpha_i)$, and $\pi_2$ is a product $\alpha_0'\alpha_1'\cdots\alpha_r'$ where $\alpha_i':Y_{i+1}'\rightarrow Y_i'$ 
are local blow ups defined by the data $(U_i',E_i',\alpha_i')$. We may assume (by composing with identity maps) that the length of each sequence is a common value $r$.
We define $J(\pi_1,\pi_2)$ by induction on $r$. Assume that $J_{r}=J(\alpha_{0}\alpha_1\cdots\alpha_{r-1},\alpha_{0}'\alpha_1'\cdots\alpha_{r-1}')$ has been constructed,
with projections $\gamma:Y_{J_r}\rightarrow Y_{r}$ and $\delta:Y_{J_r}\rightarrow Y_{r}'$. Then we define $J(\pi_1,\pi_2)$ to be the blow up
$$
J(\pi_1,\pi_2):Y_{J}=B(\mathcal I_{E_{r}}\mathcal I_{E'_{r}}\mathcal O_{J_{r}}|\gamma^{-1}(U_r)\cap \delta^{-1}(U_r'))\rightarrow Y.
$$

Suppose that $\phi:X\rightarrow Y$ is a morphism of complex analytic spaces, and $\pi:Y'\rightarrow Y\in \mathcal E(Y)$.
The morphism $\phi^{-1}[\pi]:\phi^{-1}[Y']\rightarrow X$ will denote the strict transform of $\phi$ by $\pi$ (Section 2 of \cite{HLT}).

In the case of a single local blowup $(U,E,\pi)$ of $Y$, $\phi^{-1}[Y']$ is the blow up $B(\mathcal I_E\mathcal O_{X}|\phi^{-1}(U))$.
In the case when $\pi=\pi_0\pi_{1}\cdots \pi_r$ with $\pi_i:Y_{i+1}\rightarrow Y_{i}$ given by local blow ups $(U_i,E_i,\pi_i)$, we inductively define
$\phi^{-1}[\pi]$. Assume that $\pi^{-1}[\pi_{0}\cdots\pi_{r-1}]$ has been constructed. Let $h=\pi_{0}\cdots\pi_{r-1}$, so that $\pi=h\pi_r$. 
Let $\phi':\phi^{-1}[Y_{r}]\rightarrow Y_{r}$ be the natural morphism. Then define $\phi^{-1}[Y_{r+1}]$ to be the blow up 
$B(\mathcal I_{E_r}\mathcal O_{\phi^{-1}[Y_{r}]}|(\phi')^{-1}(U_r))$.

\section{Rectilinearization}

In this section we prove rectilinearization, Theorem \ref{TheoremR}. We use the method of complexification of a real analytic morphism (Section 1, \cite{H}).

\begin{Theorem}\label{TheoremA} Suppose that $\phi:Y\rightarrow X$ is a morphism of reduced complex analytic spaces,  $K$ is a  compact  neighborhood in  $Y$  and $f$ is an \'etoile over $X$. Then there exist local monomializations 
\begin{equation}\label{eq2}
\begin{array}{ccl}
Y_i&\stackrel{\phi_i}{\rightarrow}&X_i\\
\delta_i\downarrow&&\downarrow\gamma_i\\
Y&\stackrel{\phi}{\rightarrow}&X
\end{array}
\end{equation}
for $1\le i\le r$ and $\pi:X_f\rightarrow X\in f$ such that there are commutative diagrams for $1\le i\le t\le r$
\begin{equation}\label{eq1}
\begin{array}{rcl}
\overline Y_i&\stackrel{\tau_i}{\rightarrow}&X_f\\
\beta_i\downarrow&&\downarrow\alpha_i\\
Y_i&\stackrel{\phi_i}{\rightarrow}&X_i\\
\delta_i\downarrow&&\downarrow\gamma_i\\
Y&\stackrel{\phi}{\rightarrow}&X
\end{array}
\end{equation}
where $\overline Y_i=\phi_i^{-1}[X_f]$. Let $\psi_i=\delta_i\beta_i$ and $\pi=\gamma_i\alpha_i$. There exists a closed analytic subspace $G_f$ of $X_f$ which is nowhere dense in $X_f$ such that $X_f\setminus G_f\rightarrow X$ is an open embedding, $\pi^{-1}(\pi(G_f))=G_f$, the vertical arrows are products of a finite number of local blow ups of smooth subspaces and 
$$
\cup_{i=1}^t\alpha_i^{-1}(\phi_i(\delta_i^{-1}(K)))\setminus G_f=\cup_{i=1}^t\tau_i(\psi_i^{-1}(K))\setminus G_f=\pi^{-1}(\phi(K))\setminus G_f.
$$
There exists a compact neighborhood $L$ of $f_{X_f}$ in $X_f$, a morphism of reduced complex analytic spaces $u:Z\rightarrow X$ and a  compact  neighborhood  $M$ in $Z$ such that $\dim Z<\dim Y$, $u(Z)\subset \phi(Y)$  and $\phi(K)\cap\pi(G_f\cap L)=u(M)$.
\end{Theorem}

\begin{proof}  By Theorem \ref{TheoremLM}, we may construct  local monomializations (\ref{eq2}), with relatively compact open neighborhoods $C_i$ in $Y_i$ with closures $K_i$ such that
$\{\mathcal E_{C_1},\ldots,\mathcal E_{C_r}\}$ is an open cover of the compact set $\rho_Y^{-1}(K)$ ($\rho_Y$ is proper by Theorem 3.16 \cite{H}) and $\gamma_i$ are sequences of local blow ups of smooth subspaces. We have that
\begin{equation}\label{eq3}
K\subset \cup_{i=1}^r\delta_i(C_i).
\end{equation}
Further, there exist closed analytic subspaces $G_i$ of $X_i$ which are nowhere dense in $X_i$ such that $X_i\setminus G_i\rightarrow X$ is an open embedding and $\phi_i^{-1}(G_i)$ is nowhere dense in $Y_i$.
Reindex the diagrams (\ref{eq2}) so that $f\in \mathcal E_{X_i}$ if $1\le i\le t$ and $f\not\in \mathcal E_{X_i}$ if $t<i\le r$. 
Suppose that $i$ is an index  such that  $f\not \in \mathcal E_{X_i}$ (that is, $t<i\le r$). The morphism $X_{i}\rightarrow X$ has a factorization
$$
X_i=V_n\stackrel{\sigma_n}{\rightarrow}\cdots\rightarrow V_2\stackrel{\sigma_2}{\rightarrow}V_1\stackrel{\sigma_1}{\rightarrow}V_0=X
$$
where each $\sigma_j:V_j\rightarrow V_{j-1}$ is a local blowup $(U_j,E_j,\sigma_j)$. There exists a smallest $j$ such that $f_{V_{j-1}}\not\in U_{j}$.

Let $X_i^*$ be an open  neighborhood of $f_{V_{j-1}}$ in $V_{j-1}$ which is disjoint from $U_{j}$. Then $f\in \mathcal E_{X_i^*}$ and $\mathcal E_{X_i^*}\cap \mathcal E_{X_i}=\emptyset$. Let $\pi:X_f\rightarrow X$ be a (global) resolution of singularities of the join of the $X_i$ which satisfy $f\in \mathcal E_{X_i}$ and of the $X_i^*$ such that $f\not\in \mathcal E_{X_i}$ and so that 
$\pi:X_f\rightarrow X\in f$ is a sequence of local blowups whose centers are nonsingular and such that $\alpha_i^{-1}(\phi_i(Y_i))$ is nowhere dense in $X_f$ for all $i$. Then,  $\mathcal E_{X_f}\subset \mathcal E_{X_i}$ if $f\in \mathcal E_{X_i}$  and $\mathcal E_{X_f}\cap\mathcal E_{X_i}=\emptyset$ if $f\not\in \mathcal E_{X_i}$. We have factorizations
$$
X_f\stackrel{\alpha_i}{\rightarrow}X_i\stackrel{\gamma_i}{\rightarrow} X
$$
 of $\pi$ if $f\in \mathcal E_{X_{i}}$. 
 For $1\le i\le t$ let 
 $$
 \begin{array}{rcl}
 \overline Y_i&\stackrel{\tau_i}{\rightarrow}&X_f\\
 \beta_i\downarrow&&\downarrow\alpha_i\\
 Y_i&\stackrel{\phi_i}{\rightarrow}&X_i
 \end{array}
 $$
 be the natural commutative diagram of morphisms, with $\overline Y_i=\phi_i^{-1}[X_f]$. Let $\psi_i=\delta_i\beta_i$.  Let $G_f$ be the union of the preimages of the subspaces blown up in a factorization of $\pi$ by local blowups. Then $G_f$ is a nowhere dense closed analytic subset of $X_f$ such that $X_f\setminus G_f\rightarrow X$ is an open embedding and $\pi^{-1}(\pi(G_f))=G_f$. Further, $\tau_i^{-1}(G_f)$ is nowhere dense in $\overline Y_i$ for all $i$. Let
 $U=X_f\setminus G_f$. Suppose that $q\in \phi(K)\cap \pi(U)$. There there exists $p\in K$ such that $\phi(p)=q$ and there exists $i$ and $p'\in C_i$ such that $\delta_i(p')=p$ by (\ref{eq3}). Let $q'=\phi_i(p')\in X_i$. There exists $\lambda\in \mathcal E_X$ such that $\lambda_{X_i}=q'$ and thus $\lambda_X=q$. Since $q\in \pi(U)$, we can regard $q$ as an element of $X_f$ with $\lambda_{X_f}=q$. Thus $\lambda\in \mathcal E_{X_f}$ so that $\lambda\in \mathcal E_{X_f}\cap \mathcal E_{X_i}$, and so $f\in \mathcal E_{X_i}$ as this intersection is nonempty.
 We have that $\phi_i(p')=q'$ and $\alpha_i:X_f\rightarrow X_i$ is an open embedding in a neighborhood of $q$, so $\beta_i$ is an open embedding in a neighborhood of $p'$. Thus $q=\lambda_{X_f}\in \tau_i(\psi_i^{-1}(K))$.
 Whence
 $$
 \pi^{-1}(\phi(K))\cap U\subset \cup_i\tau_i(\beta_i^{-1}(C_i\cap \delta_i^{-1}(K)))\cap U \subset \cup_i\tau_i(\psi_i^{-1}(K))\cap U.
 $$
 We have
 $$
 \cup_i\tau_i(\psi_i^{-1}(K))\subset \pi^{-1}(\phi(K))
 $$
 since
 $$
 \pi\tau_i(\psi_i^{-1}(K))=\phi\psi_i(\psi_i^{-1}(K))\subset\phi(K)
 $$
 for $1\le i\le t$.
 Thus  
 \begin{equation}\label{eq10}
 \cup_i\tau_i(\psi_i^{-1}(K))\cap U=\pi^{-1}(\phi(K))\cap U.
 \end{equation}
 
 Let $V_f$ be a relatively compact open neighborhood of $f_{X_f}$ in $X_f$. Let $L$ be the closure of $V_f$ in $X_f$. Then $\beta_i^{-1}(C_i)\cap\tau_i^{-1}(V_f)$ are relatively compact open subsets of $\overline Y_i$ with closures $L_i=\beta_i^{-1}(K_i)\cap\tau_i^{-1}(L)$ for $1\le i\le t$. Further, 
 $$
 \cup_i\phi\psi_i(\psi_i^{-1}(K)\cap L_i)\subset \pi(L)\cap \phi(K)
 $$
 and so
 $$
 \cup_i\phi\psi_i(\psi_i^{-1}(K)\cap L_i)\setminus \pi(G_f)=(\pi(L)\setminus \pi(G_f))\cap\phi(K)
 $$
 by (\ref{eq10}).
 For all $i$, the compact set $\pi(G_f)\cap[\phi\psi_i(\psi_i^{-1}(K)\cap L_i)]$ is nowhere dense in the compact set $\phi\psi_i(\psi_i^{-1}(K)\cap L_i)$ since $\tau_i^{-1}(G_f)\cap \psi_i^{-1}(K)\cap L_i$ is nowhere dense in the compact neighborhood $\psi_i^{-1}(K)\cap L_i$ in $Y_i$.
 
 Thus the compact set $\cup_i\phi\psi_i(\psi_i^{-1}(K)\cap L_i)$ is everywhere dense in the compact set $\pi(L)\cap\phi(K)$. Thus 
 $$
 \cup_i\phi\psi_i(\psi_i^{-1}(K)\cap L_i)=\pi(L)\cap\phi(K)
 $$
 and so
 $$
 \cup_i\phi\phi_i(\psi_i^{-1}(K)\cap L_i\cap \tau_i^{-1}(G_f))=\pi(G_f\cap L)\cap\phi(K).
 $$
 
  Let $Z=\coprod_{1\le i\le t}\tau_i^{-1}(G_f)$ be the disjoint union of the analytic spaces 
 $\tau_i^{-1}(G_f)$ with associated morphism $u=\coprod_i \phi\psi_i:Z\rightarrow X$ and compact subset $M=\coprod_i \psi_i^{-1}(K)\cap L_i\cap\tau_i^{-1}(G_f)$ of $Z$. 
 
 Then
 $\dim Z<\dim Y$ , $u(Z)\subset\phi(Y)$ and $\phi(K)\cap\pi(G_f\cap L)=u(M)$.
\end{proof}

Suppose $\phi:Y\rightarrow X$ is a morphism of reduced real analytic spaces such that $X$ is smooth.
Let $\tilde Y\rightarrow \tilde X$ be a complexification of $\phi:Y\rightarrow X$ such that $\tilde X$ is smooth and $\tilde Y$ is reduced. Suppose that  $\tilde K$ is a compact neighborhood in $\tilde Y$ which is invariant under the auto conjugation of $\tilde Y$. Let $K$ be the real part of $\tilde K$, which is a compact neighborhood in $Y$.   Let 
$$
\tilde Y=\tilde Y^{(n)}\supset \tilde Y^{(n-1)}\supset\cdots\supset \tilde Y^{(0)}=\emptyset
$$
be the stratification of $\tilde Y$ where $\tilde Y^{(i-1)}=\mbox{sing}(\tilde Y^{(i)})$ is the singular locus of $\tilde Y^{(i)}$, and let
$$
Y=Y^{(n)}\supset Y^{(n-1)}\supset \cdots\supset Y^{(0)}
$$
be the induced smooth real analytic stratification of $Y$. 
We have induced compact neighborhoods $\tilde K\cap Y^{(i)}$ in $Y^{(i)}$, with $K=\tilde K\cap Y^{(n)}$.  There exist global resolutions of singularities $\tilde\lambda_i:(\tilde Y^{(i)})^*\rightarrow \tilde Y^{(i)}$
 which have an auto conjugation such that the real part of $\tilde\lambda_i:(\tilde Y^{(i)})^*\rightarrow \tilde Y^{(i)}$  is $\lambda_i:(Y^{(i)})^*\rightarrow Y^{(i)}$ where $(Y^{(i)})^*$ is smooth (Desingularization I, 5.10 \cite{H}). The morphism $\tilde \lambda_i$ is proper,
  so $\tilde K_i=\tilde \lambda_i^{-1}(K\cap \tilde Y^{(i)})$ is a compact neighborhood in  $\tilde Y^{(i)}$ with compact real neighborhood  $K_i=\lambda_i^{-1}(K\cap Y^{(i)})$ in $Y^{(i)}$.

Let $Y'=\coprod_i(Y^{(i)})^*$. We have that the induced morphism $\phi^*:Y'\rightarrow Y$ is proper and surjective. Let $K'=(\lambda^*)^{-1}(K)$, a compact neighborhood in  $Y'$. 
Let $\tilde Y'=\coprod_i(\tilde Y^{(i)})^*$ with induced  complex analytic morphism $\tilde\lambda^*:\tilde Y'\rightarrow \tilde Y$. Then $\tilde\lambda^*:\tilde Y'\rightarrow \tilde Y$ is a complexification of $\lambda^*:Y'\rightarrow Y$.
Let $\tilde K'=(\lambda^*)^{-1}(\tilde K)$, which is a compact neighborhood in $\tilde Y'$ with $(\tilde K')\cap Y'=K'$.

By Theorem \ref{TheoremA}, applied to the complex analytic morphism $\tilde\phi\tilde\lambda^*:\tilde Y'\rightarrow \tilde X$, the compact neighborhood $\tilde K'$ in $\tilde Y'$  and an \'etoile $f$ over $\tilde X$,  there exist commutative diagrams
$$
\begin{array}{rcl}
\tilde{\overline Y}_i&\stackrel{\tilde\tau_i}{\rightarrow}&\tilde X_f\\
\tilde\beta_i\downarrow&&\downarrow\tilde\alpha_i\\
\tilde Y_i&\stackrel{\tilde\phi_i}{\rightarrow}&\tilde X_i\\
\tilde\delta_i\downarrow&&\downarrow\tilde\gamma_i\\
\tilde Y'&\stackrel{\tilde\phi\tilde\lambda^*}{\rightarrow}&\tilde X\\
\searrow&&\nearrow\\
&\tilde Y&
\end{array}
$$
and a closed analytic subspace $\tilde G_f$ of $\tilde X_f$ such that 
\begin{equation}\label{eq24}
\cup_{i=1}^t\tilde\tau_i(\tilde\psi_i^{-1}(\tilde K'))\setminus \tilde G_f=\tilde\pi^{-1}(\tilde\phi\tilde\lambda^*(\tilde K'))\setminus \tilde G_f
\end{equation}
and there exists a compact neighborhood $\tilde L$ of $f_{\tilde X_f}$ in $\tilde X_f$, a morphism of reduced complex analytic spaces $\tilde u:\tilde Z\rightarrow\tilde X$ and a compact neighborhood $\tilde M$ in $\tilde Z$ such that $\dim \tilde Z<\dim \tilde Y$ and
$$
\tilde\phi\tilde\lambda^*(\tilde K')\cap \tilde\pi(\tilde G_f\cap\tilde L)=\tilde u(\tilde M).
$$
We can construct the above complex analytic spaces and morphisms so that there are compatible auto conjugations which preserve $\tilde G_f$, $\tilde Z$, $\tilde L$ and $\tilde M$ and so that the real part $X_f$ of $\tilde X_f$ is nonempty if and only if $f_{\tilde X_f}$ is a real point (by Theorems 8.4 and 8.5 \cite{C}).

Taking the invariants of the auto conjugations, we thus have whenever $X_f\ne\emptyset$, induced commutative diagrams of real analytic spaces and morphisms
$$
\begin{array}{rcl}
\overline Y_i&\stackrel{\tau_i}{\rightarrow}&X_f\\
\beta_i\downarrow&&\downarrow\alpha_i\\
Y_i&\stackrel{\phi_i}{\rightarrow}& X_i\\
\delta_i\downarrow&&\downarrow\gamma_i\\
Y'&\stackrel{\phi\lambda^*}{\rightarrow}& X\\
\searrow&&\nearrow\\
&Y&
\end{array}
$$
with a closed real analytic subspace $G_f=\tilde G_f\cap X_f$ of $X_f$. We have that $G_f$ is nowhere dense in $X_f$ since $X_f$ is smooth (for instance by Lemma 8.2 \cite{C}), and thus   $\dim G_f<\dim X_f=\dim Y$.

Also, taking the real part of $\tilde u:\tilde Z\rightarrow \tilde Y'$, we have a morphism of reduced real analytic spaces $u:Z\rightarrow Y'$ and a compact subset $M$ of $Z$ such that $\dim Z<\dim Y$. 
The analog of (\ref{eq3}) in Theorem \ref{TheoremA}, $K'\subset \cup_{i=1}^r\delta_i(C_i)$ where $C_i$ is the real part of $\tilde C_i$, is true as $Y'$ is smooth (by Theorem 8.7  \cite{C} and its proof). Then the argument following (\ref{eq3}) 
in Theorem \ref{TheoremA} shows that 
$$
\cup_{i=1}^t\tau_i(\psi_i^{-1}(K'))\setminus G_f=\pi^{-1}(\phi\lambda^*(K'))\setminus G_f=\pi^{-1}(\phi(K))\setminus G_f
$$
 and
$$
\pi(K)\cap\pi(G_f\cap L)=\phi\lambda^*(K')\cap\pi(G_f\cap L)=u(M).
$$

\begin{Theorem}\label{TheoremB} Suppose that $\phi:Y\rightarrow X$ is a morphism of reduced complex analytic spaces, $K\subset Y$ is a compact neighborhood in $Y$ and $h\in \mathcal E_X$. Then there exists $d_h:X_h\rightarrow X\in h$, morphisms of reduced complex analytic spaces $\phi_i:Y_i\rightarrow X$ for $0\le i\le t$ with compact neighborhoods $K_i$ in $Y_i$ such that 
$\phi_0=\phi$, $Y_0=Y$, $K_0=K$, $\phi_{i+1}(Y_{i+1})\subset \phi_i(Y_i)$, $\dim Y_{i+1}<\dim Y_i$ for all $i$ and $Y_t=\emptyset$. There exist commutative diagrams for $0\le i\le t$
\begin{equation}\label{eq4}
\begin{array}{rcl}
\hat{Y}_{ij}&\stackrel{\sigma_{ij}}{\rightarrow}&X_h\\
b_{ij}\downarrow&&\downarrow a_i\\
\overline Y_{ij}&\stackrel{\tau_{ij}}{\rightarrow}&X_i\\
\beta_{ij}\downarrow&&\downarrow \alpha_{ij}\\
Y_{ij}&\stackrel{\phi_{ij}}{\rightarrow}&X_{ij}\\
\delta_{ij}\downarrow&&\downarrow \gamma_{ij}\\
Y_i&\stackrel{\phi_i}{\rightarrow}&X
\end{array}
\end{equation}
where $\phi_{ij}:Y_{ij}\rightarrow X_{ij}$ are monomial morphisms, $\overline Y_{ij}=\phi_{ij}^{-1}[X_i]$ and  $\hat{Y}_{ij}=\tau_{ij}^{-1}[X_h]$, $\psi_{ij}=\delta_{ij}\beta_{ij}$, $c_{ij}=\psi_{ij}b_{ij}$, $\pi_i=\gamma_{ij}\alpha_{ij}$, $\epsilon_{ij}=\alpha_{ij}a_i$  and $d_h=\pi_ia_i$ such that
\begin{equation}\label{eq9}
d_h^{-1}(\phi(K))=\cup_ia_i^{-1}[\cup_j\tau_{ij}(\psi_{ij}^{-1}(K_i))]
=\cup_{i,j}\epsilon_{ij}^{-1}(\phi_{ij}(\delta_{ij}^{-1}(K)))
\end{equation}
\end{Theorem}
 
 \begin{proof} We construct  commutative diagrams
 $$
 \begin{array}{rcl}
 \overline Y_{ij}&\rightarrow&X_i\\
 \downarrow&&\downarrow\\
 Y_{ij}&\rightarrow &X_{ij}\\
 \downarrow&&\downarrow\\
 Y_i&\rightarrow& X
 \end{array}
 $$
 satisfying the conclusions of the theorem by induction on $i$, using Theorem \ref{TheoremA}. 
 In particular, there exist nowhere dense closed analytic subsets $G_i$ of $X_i$ such that $\pi_i^{-1}(\pi_i(G_i))=G_i$ for all $i$ and $X_i\setminus G_i\rightarrow X$ is an open embedding
  and there exist compact neighborhoods $K_i$ in $Y_i$ and $L_i$ of $f_{X_i}$ in $X_i$  such that $X_t=\emptyset$, and for all $i$,
 \begin{equation}\label{eq6}
 \cup_j\tau_{ij}(\psi_{ij}^{-1}(K_i))\setminus G_i=\pi_i^{-1}(\phi_i(K_i))\setminus G_i
 \end{equation}
 and
 \begin{equation}\label{eq5}
 \phi_i(K_i)\cap\pi_i(G_i\cap L_i)=\phi_{i+1}(K_{i+1}).
 \end{equation}
 We then have (by the definition of an \'etoile) that there exists $X_h\rightarrow X\in h$ such that we have a commutative diagram (\ref{eq4}) such that $a_i(X_h)\subset L_i$ for all $i$.
 For all $i$, (\ref{eq6}) implies 
 \begin{equation}\label{eq8}
 a_i^{-1}[\cup_j\tau_{ij}(\psi_{ij}^{-1}(K_i))]\cap (X_f\setminus a_i^{-1}(G_i))
 =d_h^{-1}(\phi_i(K_i))\cap (X_f\setminus a_i^{-1}(G_i)).
 \end{equation}
 Now, (\ref{eq5}) implies 
 $$
 \pi_{i-1}^{-1}(\phi_{i-1}(K_{i-1}))\cap G_{i-1}\cap  \pi_{i-1}^{-1}(\pi_{i-1}(L_{i-1}))  =\pi_{i-1}^{-1}(\phi_i(K_i))
 $$
 for all $i$, and so since $a_{i-1}(X_h)\subset L_{i-1}$,
 $$
 d_h^{-1}(\phi_{i-1}(K_{i-1}))\cap a_{i-1}^{-1}(G_{i-1})=d_h^{-1}(\phi_i(K_i))
 $$
 for all $i$. Thus
  \begin{equation}\label{eq7}
 d_h^{-1}(\phi(K))\cap a_0^{-1}(G_0)\cap \cdots\cap a_{i-1}^{-1}(G_{i-1})=d_h^{-1}(\phi_i(K_i)).
 \end{equation}
 Since $G_t=\emptyset$,  and we certainly have
 $$
 \cup_ia_i^{-1}[\cup_j\tau_{ij}(\psi_{ij}^{-1}(K_i))]\subset d_h^{-1}(\phi(K)),
 $$
  (\ref{eq9}) follows from induction on $i$, using (\ref{eq8}) and (\ref{eq7}) and since
  $$
  \cup_{i=0}^t[a_0^{-1}(G_0)\cap \cdots \cap a_{i-1}^{-1}(G_{i-1})]\cap (X_f\setminus a_i^{-1}(G_i))=X_f.
  $$
 \end{proof}

From the discussion  after Theorem \ref{TheoremA} and Theorem \ref{TheoremB}, we obtain the following statement.
Suppose $\phi:Y\rightarrow X$ is a morphism of reduced real analytic spaces such that $X$ is smooth.
Let $\tilde Y\rightarrow \tilde X$ be a complexification of $\phi:Y\rightarrow X$ such that $\tilde X$ is smooth and $\tilde Y$ is reduced.  Suppose that $\tilde K$ is a compact neighborhood in $\tilde Y$
which is invariant under the auto conjugation of $\tilde Y$. Let $K$ be the real part of $\tilde K$ which is a compact neighborhood in $Y$.

Suppose that  $h\in \mathcal E_X$. Then there exists $\tilde d_h:\tilde X_h\rightarrow \tilde X\in h$, morphisms of reduced complex analytic spaces $\tilde \phi_i:\tilde Y_i\rightarrow \tilde X$ for $0\le i\le t$ with compact neighborhoods $\tilde K_i$ in $\tilde Y_i$ such that 
$\tilde \phi_0(\tilde K_0)= \tilde \phi(\tilde K)$, $\dim \tilde Y_{i+1}<\dim \tilde Y_i$ for all $i$ and $\tilde Y_t=\emptyset$. There exist commutative diagrams for $0\le i\le t$
\begin{equation}\label{eq20}
\begin{array}{rcl}
\tilde {\hat{Y}}_{ij}&\stackrel{\tilde \sigma_{ij}}{\rightarrow}&\tilde X_h\\
\tilde b_{ij}\downarrow&&\downarrow \tilde a_i\\
\tilde{\overline Y_{ij}}&\stackrel{\tilde\tau_{ij}}{\rightarrow}&\tilde X_i\\
\beta_{ij}\downarrow&&\downarrow \alpha_{ij}\\
\tilde Y_{ij}&\stackrel{\tilde \phi_{ij}}{\rightarrow}&\tilde X_{ij}\\
\tilde \delta_{ij}\downarrow&&\downarrow\tilde  \gamma_{ij}\\
\tilde Y_i&\stackrel{\tilde \phi_i}{\rightarrow}&\tilde X
\end{array}
\end{equation}
where $\phi_{ij}:\tilde Y_{ij}\rightarrow\tilde  X_{ij}$ are monomial morphisms, $\tilde{\overline Y}_{ij}=\tilde\phi_{ij}^{-1}[X_i]$ and  $\tilde{\hat{Y}}_{ij}=\tilde\tau_{ij}^{-1}[\tilde X_h]$. Let $\tilde\psi_{ij}=\tilde\delta_{ij}\tilde\beta_{ij}$, $\tilde c_{ij}=\tilde\psi_{ij}\tilde b_{ij}$, $\tilde\pi_i=\tilde\gamma_{ij}\tilde \alpha_{ij}$ and $\tilde d_h=\tilde\pi_i\tilde a_i$ such that
\begin{equation}\label{eq23}
\tilde d_h^{-1}(\tilde \phi(\tilde K))=\cup_{i,j}\tilde\epsilon_{ij}^{-1}(\tilde\phi_{ij}(\tilde\delta_{ij}^{-1}(\tilde K)))=\cup_i\tilde a_i^{-1}[\cup_j\tilde\tau_{ij}(\tilde\psi_{ij}^{-1}(\tilde K_i))]
\end{equation}
Further, there are compatible auto conjugations of these analytic spaces and morphisms such that the real parts are
$d_h:X_h\rightarrow X$, morphisms of reduced real analytic spaces $\phi_i:Y_i\rightarrow X$ for $0\le i\le t$ with compact neighborhoods $K_i$ in $Y_i$ with $K_i=\tilde K_i\cap Y_i$ such that 
$\phi_0(K_0)=\phi(K)$,  $\dim Y_{i+1}<\dim Y_i$ for all $i$ and $Y_t=\emptyset$. We may assume that $X_h\ne \emptyset$ if and only if $h_{\tilde X_h}$ is  a real point of $\tilde X_h$. Suppose that $X_h\ne\emptyset$. Then there exist commutative diagrams for $0\le i\le t$
\begin{equation}\label{eq21}
\begin{array}{rcl}
\hat{Y}_{ij}&\stackrel{\sigma_{ij}}{\rightarrow}&X_h\\
b_{ij}\downarrow&&\downarrow a_i\\
\overline Y_{ij}&\stackrel{\tau_{ij}}{\rightarrow}&X_i\\
\beta_{ij}\downarrow&&\downarrow a_{ij}\\
Y_{ij}&\stackrel{\phi_{ij}}{\rightarrow}&X_{ij}\\
\delta_{ij}\downarrow&&\downarrow \gamma_{ij}\\
Y_i&\stackrel{\phi_i}{\rightarrow}&X
\end{array}
\end{equation}
where $\phi_{ij}:Y_{ij}\rightarrow X_{ij}$ are monomial morphisms, $\overline Y_{ij}=\phi_{ij}^{-1}[X_i]$ and  $\hat{Y}_{ij}=\tau_{ij}^{-1}[X_h]$, $\psi_{ij}=\delta_{ij}\beta_{ij}$, $c_{ij}=\psi_{ij}b_{ij}$, $\pi_i=\gamma_{ij}\alpha_{ij}$ and $d_h=\pi_ia_i$ such that
\begin{equation}\label{eq22}
d_h^{-1}(\phi(K))=\cup_{i,j}\epsilon_{ij}^{-1}(\phi_{ij}(\delta_{ij}^{-1}( K)))=\cup_i a_i^{-1}[\cup_j\tau_{ij}(\psi_{ij}^{-1}( K_i))]
\end{equation}

\begin{Theorem}\label{TheoremC} Suppose that $X$ and $Y$ are real analytic spaces such that $X$ is smooth and $\phi:Y\rightarrow X$ is a proper real analytic map. Let $p\in X$. Then there exists a finite number of real analytic maps $\pi_{\alpha}:V_{\alpha}\rightarrow X$ such that:
\begin{enumerate}
\item[1)] Each $V_{\alpha}$ is smooth and each $\pi_{\alpha}$ is a composition of local blowups of nonsingular sub varieties,
\item[2)] There exist  compact neighborhoods $N_{\alpha}$ in $V_{\alpha}$ for all $\alpha$ such that  $\cup_{\alpha}\pi_{\alpha}(N_{\alpha})$ is a compact neighborhood of $p$ in $X$,
\item[3)] For all $\alpha$,  $\pi_{\alpha}^{-1}(\phi(Y))$ is a semi analytic subset of $V_{\alpha}$.
\end{enumerate}
\end{Theorem}

\begin{proof} 
Let $\tilde\phi:\tilde Y\rightarrow \tilde X$ be a complexication of $\phi:Y\rightarrow X$ so that  $\tilde X$ is smooth and $\tilde Y$ is reduced.

Let $\tilde U$ be a relatively compact open neighborhood of $p$ in $\tilde X$ which is invariant under the auto conjugation of $\tilde X$ and let $\tilde L$ be the closure of $\tilde U$ in $\tilde X$. Let $L=\tilde L\cap X$, a compact neighborhood of $p$ in $X$. 
Let $K'=\tilde\phi^{-1}(\tilde L)$. The real part of $K'$ is $K=\phi^{-1}(L)$ which is compact since $\phi$ is proper.  Let $N$ be a compact neighborhood of $K$ in $\tilde Y$ which contains $K$ and is preserved by the auto conjugation of $\tilde Y$. Let $\tilde K=K'\cap N$. The set $\tilde K$ is a compact neighborhood in $\tilde Y$ which is preserved by the auto conjugation of $\tilde Y$ such that the real part of $\tilde K$ is $K$.
Let $U$ be the real part of $\tilde U$ which is an open neighborhood of $p$ in $X$ with closure $L$ in $X$. Let $V=\phi^{-1}(U)$, whose closure is  $K=\phi^{-1}(L)$. We have that
\begin{equation}\label{eq30}
\phi(V)=\phi(K)\cap U.
\end{equation}

For each $h\in \mathcal E_{\tilde X}$ such that $X_h\ne\emptyset$, we have 
associated complex analytic morphisms $\tilde d_h:\tilde X_h\rightarrow \tilde X$ with real part $d_h:X_h\rightarrow X$, and associated diagrams (\ref{eq20}) with real part (\ref{eq21}). 
For all $i,j$, we have that
$$
d_h^{-1}(\phi(K))=d_h^{-1}(\phi(K))\cap[\cup_{i,j}\epsilon_{ij}^{-1}(\phi_{ij}(Y_{ij}))]
$$
by (\ref{eq22}).  Thus
\begin{equation}\label{eq31}
\begin{array}{lll}
d_h^{-1}(\phi(V))&=&d_h^{-1}(\phi(K)\cap U)=d_h^{-1}(\phi(K))\cap d_h^{-1}(U)\\
&=& d_h^{-1}(\phi(K))\cap [\cup_{i,j}\epsilon_{ij}^{-1}(\phi_{ij}(Y_{ij}))]\cap d_h^{-1}U)\\&=&d_h^{-1}(U)\cap[\cup_{i,j}\epsilon_{ij}^{-1}(\phi_{ij}(Y_{ij}))].
\end{array}
\end{equation}

We now establish that $d_h^{-1}(\phi(V))$ is a semi analytic subset of $X_h$. For all $i,j$, $\phi_{ij}( Y_{ij})$ is semi analytic in $X_{ij}$ since $\phi_{ij}$ is a monomial morphism (by the Tarski Seidenberg  theorem, c.f. Theorem1.5 \cite{BM3}. Thus $d_h^{-1}(\phi(V))$ is semianalytic in $X_h$ by (\ref{eq31}).

For $h\in \mathcal E_{\tilde X}$, let $\tilde C_h$ be an open relatively compact neighborhood of $h_{\tilde X}$ in $\tilde X_h$ on which the auto conjugation acts. Let $\overline d_h:\tilde C_h\rightarrow \tilde X$ be the induced morphism. Let $C$ be a compact neighborhood of $p$ in $\tilde X$ such that $C\subset \tilde U$ and let $C'=\rho_{\tilde X}^{-1}(C)$, The set $C'$ is compact since $\rho_{\tilde X}$ is proper (Theorem 3.4 \cite{Hcar} or Theorem 3.16 \cite{H}). The open sets $\mathcal E_{\overline d_f}$ for $f\in C'$ give an open cover of $C'$, so there is a finite sub cover, which we index as 
$\mathcal E_{\overline d_{f_1}},\ldots,\mathcal E_{\overline d_{f_t}}$. 
We may replace $\tilde X_{f_i}$ with $\tilde d_{f_i}^{-1}(\tilde U)$, so that (\ref{eq31}) implies that
$$
d_h^{-1}(\phi(Y))=\cup_{i,j}\epsilon_{ij}^{-1}(\phi_{ij}(Y_{ij}))
$$
 is a semi analytic set.

Let $C_{f_i}$ be the closure of $\tilde C_{f_i}$ in $\tilde X_{\tilde d_{f_i}}$ which is compact. Since $\rho_{\tilde X}$ is surjective and continuous, we have inclusions of compact sets 
$$
p\in C\subset\cup_{i=1}^t\tilde d_{f_i}(C_{f_i}).
$$
Since $X$ and $\tilde X$ are smooth and each $\tilde d_{f_i}$ is a finite product of local blowups of closed analytic subspaces which are preserved by the auto conjugation, if $F_{f_i}$ is the the union of
the preimage in $\tilde X_{f_i}$ of these subspaces, then $F_{f_i}$ is a nowhere dense closed analytic subspace of $\tilde X_{f_i}$ which is preserved by the auto conjugation such that 
$\tilde X_{f_i}\setminus F_{f_i} \rightarrow \tilde X$ is an open embedding. The image $\tilde d_{f_i}(F_{f_i})$ is nowhere dense in $\tilde X$, and since $\tilde X$ and $X$ are smooth varieties,
 $\tilde d_{f_i}(F_{f_i})\cap X$ is nowhere dense in $X$ (as  in the proof of Theorem 8.7 \cite{C}).

Let $C^*=C\cap X$ which is a compact neighborhood of $p$ in $X$ which is contained in $L$. Let $p'\in C^*\setminus \cup_{i=1}^t\tilde d_{f_i}(F_{f_i})$. Then there exist $i$ and $e\in \mathcal E_{\overline d_{f_i}}$ such that $e_{\tilde X}=p'$. Let $p_i=e_{\tilde X_{f_i}}\in C_i\subset \tilde X_{f_i}$. Since $p_i\not\in F_{f_i}$, $\tilde d_{f_i}$ is an open embedding near $p_i$, and since $p'$ is real, $p_i\in X_{f_i}$ is real. Thus $p'\in d_{f_i}(C_{f_i}\cap X_{f_i})$. We thus have that the set 
$C^*\setminus \cup_{i=1}^t\tilde d_{f_i}(F_{f_i})$, which we have shown is dense in $C^*$, is contained in the compact set $\cup_{i=1}^t d_{f_i}(C_{f_i}\cap X_{f_i})$. Thus its closure $C^*$ is contained in $\cup_{i=1}^t d_{f_i}(C_{f_i}\cap X_{f_i})$, giving the conclusion of 2) of the theorem.

\end{proof}

\begin{Theorem}\label{TheoremD} Suppose that $X$  is a smooth real analytic space. Suppose that $A$ is a sub analytic subset of $X$ and  that $p\in X$. Then there exists a finite number of real analytic maps $\pi_{\alpha}:V_{\alpha}\rightarrow X$ such that:
\begin{enumerate}
\item[1)] Each $\pi_{\alpha}$ is a composition of local blowups of nonsingular sub varieties,
\item[2)] There exist compact neighborhoods  $N_{\alpha}$ in $V_{\alpha}$ for all $\alpha$ such that  $\cup_{\alpha}\pi_{\alpha}(N_{\alpha})$ is a compact neighborhood of $p$ in $X$,
\item[3)] For all $\alpha$,  $\pi_{\alpha}^{-1}(A)$ is a semianalytic subset of $V_{\alpha}$.
\end{enumerate}
\end{Theorem}

\begin{proof} After replacing $X$ with a suitable open neighborhood of $p$, we have by Definition 6.10 \cite{H}, an expression
$$
A=\cup_{k\in I}\cap_{l\in J}(A_{kl}\setminus B_{kl})
$$
where $I$ and $J$ are nonempty finite index sets and there are proper real analytic maps of reduced analytic spaces
$\phi_{kl}:Y_{kl}\rightarrow X$ and $\psi_{kl}:Z_{kl}\rightarrow X$ such that $A_{kl}=\phi_{kl}(Y_{kl})$ and $B_{kl}=\psi_{kl}(Z_{kl})$. 
Let $\tilde X$ be a smooth complexification of $X$ and let $\tilde\phi_{kl}:\tilde Y_{kl}\rightarrow \tilde X$ and $\tilde\psi_{kl}:\tilde Z_{kl}\rightarrow\tilde  X$ be complexifications of $\phi_{kl}$ and $\psi_{kl}$.
For each $h\in \mathcal E_{\tilde X}$, $k\in I$ and  $l\in J$ we construct as in the proof of Theorem \ref{TheoremC} complex analytic morphisms of smooth analytic spaces
$\tilde d_h^{kl}:(\tilde X_h)_{kl}\rightarrow \tilde X\in h$ and $(\tilde d'_h)^{kl}:(\tilde X'_h)_{kl}\rightarrow \tilde X\in h$  with auto conjugations, so that taking the invariants of the auto conjugations we have real analytic morphisms
$d_h^{kl}:(X_h)_{k,l}\rightarrow X$ and $(d_h')^{kl}:(X_h')_{kl}\rightarrow X$ such that $(d_h^{kl})^{-1}(A_{kl})$ is semianalytic in $(X_h)_{kl}$ and $((d_h')^{kl})^{-1}(B_{kl})$ is semianalytic in $(X_h')_{kl}$ for $k\in I$ and $l\in J$.
There exists $\tilde d_h:\tilde X_h\rightarrow \tilde X\in h$ with auto conjugation such that there are factorizations  $\tilde\beta_{kl}:\tilde X_h\rightarrow (\tilde X_h)_{kl}$ and $\tilde\gamma_{kl}:\tilde X_h\rightarrow (\tilde X_h')_{kl}$ 
for all $k\in I$ and $l\in J$. Thus taking the real part of  $\tilde d_h:\tilde X_h\rightarrow \tilde X$, we have real analytic morphisms
$\beta_{kl}:X_h\rightarrow (X_h)_{kl}$ and $\gamma_{kl}:X_h\rightarrow (X_h')_{kl}$ factoring through $X_h\rightarrow X$. Thus 
$$
d_h^{-1}(A_{kl})=\beta_{kl}^{-1}(d_h^{kl})^{-1}(A_{kl})
$$
and
$$
d_h^{-1}(B_{kl})=\gamma_{kl}^{-1}((d_h')^{kl})^{-1}(B_{kl})
$$
are semi analytic in $X_h$. 

We now proceed as in the proof of Theorem \ref{TheoremC} to obtain the condition 2) of Theorem \ref{TheoremD}.
\end{proof}

\begin{Theorem}\label{TheoremR} Let $X$ be a smooth connected real analytic space and let $A$ be a sub analytic subset of $X$. Let $p\in X$ and let $n=\dim X$. Then there exist a finite number of real analytic morphisms $\pi_{\alpha}:V_{\alpha}\rightarrow X$
which are finite sequences of local blowups over $X$ and  induce an open embedding of an open dense subset of $V_{\alpha}$ into $X$ such that:
\begin{enumerate}
\item[1)] Each $V_{\alpha}$ is isomorphic to $\RR^n$,
\item[2)] There exist compact neighborhoods $K_{\alpha}$ in $V_{\alpha}$ such that $\cup_{\alpha}(K_{\alpha})$ is a compact neighborhood of $p$ in $X$,
\item[3)] For each $\alpha$, $\pi_{\alpha}^{-1}(A)$ is union of quadrants in $\RR^n$.
\end{enumerate}
\end{Theorem}

\begin{proof} The proof follows from Theorem \ref{TheoremD} and Proposition 7.2 \cite{H} and Lemma 7.2.1 \cite{H}.
\end{proof}

\end{document}